\documentclass[11pt]{article}

\usepackage[T1]{fontenc}
\usepackage[a4paper,margin=1in]{geometry}
\usepackage{amsmath,amssymb,mathtools}
\usepackage{amsthm}
\usepackage{newtxtext,newtxmath}
\usepackage{microtype}
\usepackage[numbers,sort&compress]{natbib}
\usepackage[colorlinks=true,linkcolor=black,citecolor=black,urlcolor=black]{hyperref}
\usepackage[nameinlink,capitalize,noabbrev]{cleveref}

\makeatletter
\g@addto@macro\UrlBreaks{\do\0\do\1\do\2\do\3\do\4\do\5\do\6\do\7\do\8\do\9}
\makeatother

\allowdisplaybreaks
\numberwithin{equation}{section}
\newtheorem{theorem}{Theorem}[section]
\newtheorem{proposition}[theorem]{Proposition}
\newtheorem{lemma}[theorem]{Lemma}
\newtheorem{corollary}[theorem]{Corollary}
\theoremstyle{definition}
\newtheorem{definition}[theorem]{Definition}
\theoremstyle{remark}
\newtheorem{remark}[theorem]{Remark}

\crefname{theorem}{Theorem}{Theorems}
\crefname{lemma}{Lemma}{Lemmas}
\crefname{proposition}{Proposition}{Propositions}
\crefname{corollary}{Corollary}{Corollaries}
\crefname{definition}{Definition}{Definitions}
\crefname{remark}{Remark}{Remarks}
\crefname{equation}{equation}{equations}
\crefname{section}{Section}{Sections}
\crefname{appendix}{Appendix}{Appendices}

\newcommand{\R}{\mathbb R}
\newcommand{\cS}{\mathcal S}
\newcommand{\Om}{\Omega}
\newcommand{\norm}[1]{\lVert #1\rVert}
\newcommand{\doi}[1]{\href{https://doi.org/#1}{doi:\nolinkurl{#1}}}

\hypersetup{
  pdftitle={R-Linear Convergence of Barzilai--Borwein Methods on Strictly Convex Quadratics via Closed-Cone Homogeneous Dynamics},
  pdfauthor={Shutai Yang and Ya-xiang Yuan},
  pdfkeywords={Barzilai--Borwein method, spectral gradient method, positively homogeneous dynamics, closed cone, R-linear convergence}
}

\begin{document}

\title{$R$-Linear Convergence of Barzilai--Borwein Methods on Strictly Convex Quadratics via Closed-Cone Homogeneous Dynamics}

\author{Shutai Yang\textsuperscript{1,2}\qquad
Ya-xiang Yuan\textsuperscript{1}\\[0.6em]
\small\textsuperscript{1}State Key Laboratory of Scientific and Engineering Computing,\\
\small Academy of Mathematics and Systems Science, Chinese Academy of Sciences,\\
\small and University of Chinese Academy of Sciences, Beijing, China\\
\small\textsuperscript{2}School of Mathematical Sciences, University of Science and Technology of China, Hefei, China\\[0.3em]
\small Shutai Yang: \texttt{yangshutai26@mails.ucas.ac.cn}\\
\small Ya-xiang Yuan: \texttt{yyx@lsec.cc.ac.cn}}

\date{}

\maketitle

\begin{abstract}
We give an $R$-linear convergence proof for BB1, BB2, and fixed
positive spectral-weight variants on finite-dimensional strictly
convex quadratics.  The delayed recurrence is written as a first-order
system on the closed cone of compatible consecutive-gradient pairs.
Its transition is continuous and positively homogeneous of degree
one, including at finite termination.  A compactness theorem shows
that pointwise convergence of every cone orbit implies a uniform
finite-step contraction and hence a global $R$-linear estimate.  By a
realization lemma, every compatible state initiates a BB1 trajectory,
so Raydan's theorem yields pointwise stability on the entire cone.
Finally, the
transformation $g\mapsto W^{1/2}g$ conjugates every fixed positive
spectral-weight rule to BB1.  Consequently, BB1, BB2, and all such
weighted rules have the same homogeneous growth radius.
\end{abstract}



\section{Introduction}\label{sec:introduction}

Consider the strictly convex quadratic problem
\begin{equation}\label{eq:intro-problem}
  \min_{x\in\R^n} f(x)
  =\frac12x^THx-b^Tx,
  \qquad H=H^T\succ0,
\end{equation}
with minimizer $x_*=H^{-1}b$ and gradients
$g_k=H(x_k-x_*)$.  A gradient step satisfies
\begin{equation}\label{eq:intro-gradient-update}
  x_{k+1}=x_k-\alpha_k g_k,
  \qquad
  g_{k+1}=(I-\alpha_kH)g_k.
\end{equation}
The two classical Barzilai--Borwein (BB) rules choose, on a quadratic,
\begin{equation}\label{eq:intro-BB-steps}
  \alpha_k^{\rm BB1}
  =\frac{g_{k-1}^Tg_{k-1}}{g_{k-1}^THg_{k-1}},
  \qquad
  \alpha_k^{\rm BB2}
  =\frac{g_{k-1}^THg_{k-1}}{g_{k-1}^TH^2g_{k-1}}.
\end{equation}
Their defining feature is the one-step delay: the stepsize applied to
$g_k$ is determined by $g_{k-1}$.

Barzilai and Borwein introduced these rules and analyzed their
two-dimensional behavior \cite{BarzilaiBorwein1988}. Raydan proved
global convergence of both choices on finite-dimensional strictly
convex quadratics. Dai and Liao established $R$-linear convergence in
arbitrary dimension for BB1 and recorded the analogous BB2 conclusion
\cite{Raydan1993,DaiLiao2002}. The latter proof strengthens
a spectral-component induction into a uniform finite-window estimate.
Related coordinatewise mechanisms underlie the bounded-retard
framework of Friedlander, Mart\'inez, Molina, and Raydan and the
Property~A analysis of Dai
\cite{FriedlanderMartinezMolinaRaydan1999,Dai2003}.

In this paper, the rate is obtained from the geometry of an augmented state.  The state is a pair of consecutive gradients.  Its realizable values form the cone
\[
  (u,v),\qquad
  v=(I-\alpha H)u,\qquad
  \alpha\in[\lambda_{\max}(H)^{-1},
            \lambda_{\min}(H)^{-1}].
\]
The cone is closed and forward invariant, and its only state with
first component zero is the origin.  On this cone the delayed BB
update is continuous, including at finite termination, and positively
homogeneous of degree one.

The abstract step is a finite-dimensional compactness principle:
pointwise convergence of every orbit of a continuous homogeneous
self-map on a closed cone is equivalent to a uniform $R$-linear
estimate.  Related stability principles for homogeneous systems are
studied in \cite{Tuna2008,ShenHu2012}.  A realization lemma makes
Raydan's theorem applicable from every compatible state, providing the
pointwise input for BB1.  For a fixed positive spectral weight
$W=\omega(H)$, the map $(u,v)\mapsto(W^{1/2}u,W^{1/2}v)$ conjugates the
weighted pair dynamics to BB1.  This transfers one decay factor to
BB2 and to every fixed weighted rule and identifies their homogeneous
growth radii.

\Cref{sec:homogeneous} proves the compactness principle and records its
growth radius.  \Cref{sec:weighted-BB} constructs the compatible cone,
establishes the spectral conjugacy, and derives estimates for gradient
pairs, iterates, and objective gaps.  A complementary sharp-factor
analysis appears in \cite{YangYuanSharp2026}.

\section{Homogeneous dynamics on closed cones}\label{sec:homogeneous}

Let $X$ be a finite-dimensional normed space.  A nonempty set
$\Om\subset X$ is called a \emph{closed cone} if it is closed,
$0\in\Om$, and
\[
  x\in\Om,\quad t\ge0
  \quad\Longrightarrow\quad
  tx\in\Om.
\]
A map $T:\Om\to\Om$ is positively
homogeneous of degree one if
\begin{equation}\label{eq:positive-homogeneity}
  T(tx)=tT(x),
  \qquad t\ge0,\quad x\in\Om.
\end{equation}
Positive homogeneity implies $T(0)=0$.

\begin{definition}\label{def:Rlinear-map}
The origin is \emph{globally $R$-linearly stable} for $T$ on $\Om$ if
there exist $C\ge1$ and $\rho\in(0,1)$ such that
\begin{equation}\label{eq:Rlinear-map}
  \norm{T^k(x)}
  \le C\rho^k\norm{x},
  \qquad x\in\Om,\quad k\ge0.
\end{equation}
\end{definition}

\begin{theorem}\label{thm:closed-cone}
Let $\Om$ be a closed cone in $X$, and let $T:\Om\to\Om$ be continuous
and positively homogeneous of degree one.  The following statements
are equivalent.
\begin{enumerate}
\item[\textnormal{(i)}]
For every $x\in\Om$, one has $T^k(x)\to0$.
\item[\textnormal{(ii)}]
There exist $C\ge1$ and $\rho\in(0,1)$ for which
\eqref{eq:Rlinear-map} holds.
\item[\textnormal{(iii)}]
There exist an integer $N\ge1$ and $\theta\in(0,1)$ such that
\begin{equation}\label{eq:fixed-horizon-contraction}
  \norm{T^N(x)}\le\theta\norm{x},
  \qquad x\in\Om.
\end{equation}
\end{enumerate}
\end{theorem}

\begin{proof}
Suppose \textnormal{(iii)} holds and set
\[
  \Sigma:=\{x\in\Om:\norm{x}=1\}.
\]
If $\Om=\{0\}$, there is nothing to prove.  Otherwise $\Sigma$ is
compact.  Continuity and positive homogeneity give
\[
  L_N:=
  \max_{0\le j<N}\ \max_{y\in\Sigma}\norm{T^j(y)}<\infty.
\]
Write $k=qN+r$, with $0\le r<N$.  Iterating
\eqref{eq:fixed-horizon-contraction} yields
\[
  \norm{T^k(x)}
  \le L_N\theta^q\norm{x}
  \le \frac{L_N}{\theta}
       \bigl(\theta^{1/N}\bigr)^k\norm{x}.
\]
Thus \textnormal{(ii)} holds.  The implication
\textnormal{(ii)}$\Rightarrow$\textnormal{(i)} is immediate.

It remains to prove
\textnormal{(i)}$\Rightarrow$\textnormal{(ii)}.  Fix
$\eta\in(0,1)$.  For each $x\in\Sigma$, pointwise convergence gives an
integer $n_x\ge1$ such that
\[
  \norm{T^{n_x}(x)}<\eta.
\]
By continuity of $T^{n_x}$, there is a relative neighborhood $U_x$
of $x$ in $\Sigma$ on which
\[
  \norm{T^{n_x}(y)}\le\eta.
\]
Choose a finite subcover
$U_{x_1},\ldots,U_{x_s}$ and set
\[
  N_0:=\max_{1\le i\le s}n_{x_i},
  \qquad
  L:=\max_{0\le j\le N_0}\ \max_{y\in\Sigma}
       \norm{T^j(y)}.
\]
For every $y\in\Sigma$, choose one
$m(y)\in\{1,\ldots,N_0\}$ such that
\begin{equation}\label{eq:variable-contraction}
  \norm{T^{m(y)}(y)}\le\eta.
\end{equation}

Fix $x\in\Om\setminus\{0\}$.  Set $z_0=x$ and $s_0=0$.  As long as
$z_q\ne0$, define
\[
  m_q:=m\left(\frac{z_q}{\norm{z_q}}\right),
  \qquad
  s_{q+1}:=s_q+m_q,
  \qquad
  z_{q+1}:=T^{m_q}(z_q).
\]
Positive homogeneity and \eqref{eq:variable-contraction} give
\begin{equation}\label{eq:greedy-decay}
  \norm{z_{q+1}}\le\eta\norm{z_q},
  \qquad
  \norm{z_q}\le\eta^q\norm{x}.
\end{equation}
If some $z_q$ vanishes, all later iterates vanish.  Otherwise
$s_q\ge q$ and $s_q\to\infty$.  For every iterate that has not already
vanished, choose the unique $q$ such that
$s_q\le k<s_{q+1}$.  Since $k-s_q\le N_0$,
\[
  \norm{T^k(x)}
  =\norm{T^{k-s_q}(z_q)}
  \le L\norm{z_q}
  \le L\eta^q\norm{x}.
\]
Moreover,
\[
  k<s_{q+1}\le(q+1)N_0,
  \qquad\text{so}\qquad
  q>\frac{k}{N_0}-1.
\]
For every $k\ge1$, consequently,
\begin{equation}\label{eq:compactness-rate}
  \norm{T^k(x)}
  \le\frac{L}{\eta}
       \bigl(\eta^{1/N_0}\bigr)^k\norm{x}.
\end{equation}
The estimate is trivial after finite termination.  Enlarging the
prefactor to at least one also covers $k=0$, proving
\textnormal{(ii)}.

Finally, if \textnormal{(ii)} holds, choose $N$ large enough that
$C\rho^N<1$ and set $\theta=C\rho^N$.  This proves
\textnormal{(ii)}$\Rightarrow$\textnormal{(iii)} and completes the
equivalence.
\end{proof}

\subsection{The homogeneous growth radius}\label{subsec:growth-radius}

For a nontrivial closed cone and a continuous positively homogeneous
self-map, set
\[
  \Sigma_\Om:=\{x\in\Om:\norm{x}=1\}
\]
and define
\begin{equation}\label{eq:ak-definition}
  a_k(T):=\max_{x\in\Sigma_\Om}\norm{T^k(x)},
  \qquad k\ge0.
\end{equation}
The maximum exists by compactness, and positive homogeneity gives
\begin{equation}\label{eq:submultiplicative-ak}
  a_{k+\ell}(T)\le a_k(T)a_\ell(T).
\end{equation}
Indeed, for $x\in\Sigma_\Om$ with $T^\ell(x)\ne0$, normalize
$T^\ell(x)$ and use homogeneity to obtain
\[
  \norm{T^{k+\ell}(x)}
  \le a_k(T)\norm{T^\ell(x)}
  \le a_k(T)a_\ell(T);
\]
the case $T^\ell(x)=0$ is immediate.
Hence the following limit exists:
\begin{equation}\label{eq:growth-radius}
  r_\Om(T):=
  \lim_{k\to\infty}a_k(T)^{1/k}
  =\inf_{k\ge1}a_k(T)^{1/k}.
\end{equation}

\begin{corollary}\label{cor:growth-radius}
Under the assumptions of \cref{thm:closed-cone}, pointwise convergence
is equivalent to $r_\Om(T)<1$. Whenever these equivalent stability
conditions hold, $r_\Om(T)$ is the infimum of all $\rho\in(0,1)$ for
which an estimate \eqref{eq:Rlinear-map} holds with some finite
prefactor $C$.
\end{corollary}

\begin{proof}
Pointwise convergence implies \eqref{eq:Rlinear-map} by
\cref{thm:closed-cone}, hence
$a_k(T)\le C\rho^k$ and $r_\Om(T)\le\rho<1$.
Conversely, if $r_\Om(T)<1$, choose
$\rho\in(r_\Om(T),1)$.  By \eqref{eq:growth-radius}, there is $N$ for
which $a_N(T)<\rho^N$.  Writing $k=qN+r$ and using
\eqref{eq:submultiplicative-ak} gives
\[
  a_k(T)
  \le a_N(T)^q\max_{0\le j<N}a_j(T)
  \le C_\rho\rho^k
\]
for a finite $C_\rho$.  Positive homogeneity gives
\eqref{eq:Rlinear-map}.  The same argument proves that every
$\rho>r_\Om(T)$ is admissible, whereas any admissible rate bounds the
limit in \eqref{eq:growth-radius}.
\end{proof}

\begin{remark}\label{rem:finite-dimensionality}
The compact normalized section is essential.  On $\ell^2$, let
\[
  Te_j=\left(1-\frac1j\right)e_j.
\]
Then $T$ is linear and $T^kx\to0$ for every $x\in\ell^2$, but
$\norm{T^k}=1$ for every $k$; no uniform $R$-linear estimate holds.
Thus \cref{thm:closed-cone} does not extend to infinite-dimensional
spaces without an additional compactness or uniformity hypothesis.
\end{remark}

\section{BB dynamics on compatible gradient pairs}\label{sec:weighted-BB}

Translate the minimizer in \eqref{eq:intro-problem} to the origin and
write
\begin{equation}\label{eq:BB-quadratic}
  f(x)=\frac12x^THx,
  \qquad
  g=Hx,
  \qquad
  0<a:=\lambda_{\min}(H)\le
  \lambda_{\max}(H)=:b.
\end{equation}
A \emph{positive spectral weight} is
\[
  W=\omega(H),
  \qquad
  \omega(\lambda)>0
  \quad\text{for every }\lambda\in\sigma(H).
\]
Thus $W\succ0$ and $WH=HW$. Write
\[
  \kappa_2(W):=\norm{W}_2\norm{W^{-1}}_2.
\]
For $u\ne0$, define the delayed weighted
Rayleigh step
\begin{equation}\label{eq:weighted-Rayleigh-step}
  \alpha_W(u):=
  \frac{u^TWu}{u^TWHu}.
\end{equation}
The generalized Rayleigh quotient gives
\begin{equation}\label{eq:weighted-step-bounds}
  \frac1b\le\alpha_W(u)\le\frac1a.
\end{equation}
The choices $W=I$ and $W=H$ are BB1 and BB2, respectively.
Starting from an arbitrary positive finite first step $\alpha_0$, the
weighted delayed recurrence is
\begin{equation}\label{eq:weighted-BB-recurrence}
  g_{k+1}=(I-\alpha_kH)g_k,
  \qquad
  \alpha_k=\alpha_W(g_{k-1}),
  \qquad k\ge1,
\end{equation}
with the usual zero continuation after termination.

\subsection{The closed invariant cone}\label{subsec:BB-cone}

Let
\begin{equation}\label{eq:BB-cone}
  \Om_{\rm BB}(H):=
  \left\{(u,v)\in\R^n\times\R^n:
  v=(I-\alpha H)u
  \text{ for some }\alpha\in[b^{-1},a^{-1}]
  \right\}
\end{equation}
and equip the product space with
\[
  \norm{(u,v)}_\times
  :=\bigl(\norm{u}^2+\norm{v}^2\bigr)^{1/2}.
\]
Define
\begin{equation}\label{eq:BB-map}
  T_W(u,v):=
  \begin{cases}
  \bigl(v,(I-\alpha_W(u)H)v\bigr),&u\ne0,\\[0.2em]
  (0,0),&(u,v)=(0,0).
  \end{cases}
\end{equation}
There is no other state with $u=0$ in \eqref{eq:BB-cone}.

\begin{proposition}\label{prop:BB-map}
The set $\Om_{\rm BB}(H)$ is a closed cone.  For every fixed positive
spectral weight $W$, the transition
\[
  T_W:\Om_{\rm BB}(H)\longrightarrow\Om_{\rm BB}(H)
\]
is a continuous positively homogeneous self-map.
\end{proposition}

\begin{proof}
Closure under nonnegative scaling follows immediately from
\eqref{eq:BB-cone}.  Let
$(u^\ell,v^\ell)\to(u,v)$ with
$(u^\ell,v^\ell)\in\Om_{\rm BB}(H)$.  Choose
$\alpha_\ell\in[b^{-1},a^{-1}]$ such that
\[
  v^\ell=(I-\alpha_\ell H)u^\ell.
\]
The interval is compact, so a subsequence satisfies
$\alpha_\ell\to\alpha\in[b^{-1},a^{-1}]$.  Passing to the limit gives
$v=(I-\alpha H)u$, and the cone is closed.

If $(u,v)\in\Om_{\rm BB}(H)$ and $u\ne0$, then
\eqref{eq:weighted-step-bounds} shows that
\[
  (I-\alpha_W(u)H)v
\]
is a legal one-step update of $v$.  Hence $T_W(u,v)$ belongs to the
same cone.  The zero state maps to itself.

For $t>0$, \eqref{eq:weighted-Rayleigh-step} gives
$\alpha_W(tu)=\alpha_W(u)$.  Thus
\[
  T_W(tu,tv)=tT_W(u,v).
\]
The identity also holds at $t=0$.  Continuity away from the origin is
immediate.  At the origin, set
\[
  M_H:=\max_{\alpha\in[b^{-1},a^{-1}]}
       \norm{I-\alpha H}_2<\infty.
\]
For every nonzero compatible state,
\[
  \norm{(I-\alpha_W(u)H)v}
  \le M_H\norm{v}.
\]
Therefore $T_W(u,v)\to(0,0)$ whenever
$(u,v)\to(0,0)$ within the cone.  This proves continuity.
\end{proof}

\subsection{Spectral conjugacy}\label{subsec:BB-conjugacy}

Set $S=W^{1/2}$ and define
\begin{equation}\label{eq:BB-product-conjugacy}
  \cS_W(u,v):=(Su,Sv).
\end{equation}
Since $S$ commutes with $H$, $\cS_W$ maps $\Om_{\rm BB}(H)$ onto itself.

\begin{proposition}\label{prop:BB-conjugacy}
Let $T_I$ denote the BB1 transition.  Then
\begin{equation}\label{eq:BB-conjugacy}
  \cS_WT_W=T_I\cS_W,
  \qquad
  T_W=\cS_W^{-1}T_I\cS_W.
\end{equation}
In particular,
\begin{equation}\label{eq:BB-growth-radius-conjugate}
  r_{\Om_{\rm BB}(H)}(T_W)=r_{\Om_{\rm BB}(H)}(T_I).
\end{equation}
\end{proposition}

\begin{proof}
For $u\ne0$,
\[
  \alpha_I(Su)
  =\frac{(Su)^T(Su)}{(Su)^TH(Su)}
  =\frac{u^TWu}{u^TWHu}
  =\alpha_W(u).
\]
Multiplying both components of \eqref{eq:BB-map} by $S$ proves
\eqref{eq:BB-conjugacy}. Iteration gives
$T_W^k=\cS_W^{-1}T_I^k\cS_W$. If $a_k(T)$ denotes
\eqref{eq:ak-definition} on $\Om_{\rm BB}(H)$ and
$\kappa(\cS_W)=\norm{S}_2\norm{S^{-1}}_2$, then
\begin{equation}\label{eq:BB-ak-conjugacy}
  \kappa(\cS_W)^{-1}a_k(T_I)
  \le a_k(T_W)
  \le \kappa(\cS_W)a_k(T_I).
\end{equation}
Indeed, the upper bound follows from the conjugacy and homogeneity
after normalizing $\cS_Wz$; the lower bound follows by interchanging
$T_W$ and $T_I$. Taking $k$th roots in
\eqref{eq:BB-ak-conjugacy} proves
\eqref{eq:BB-growth-radius-conjugate}.
\end{proof}

\subsection{State-uniform \texorpdfstring{$R$}{R}-linear convergence}

\begin{lemma}\label{lem:BB-realization}
Every state $(u,v)\in\Om_{\rm BB}(H)$ is the consecutive-gradient
state obtained after one admissible warm-up step.  Starting from this
state, its subsequent delayed weighted Rayleigh trajectory is exactly
the orbit of $T_W$.
\end{lemma}

\begin{proof}
If $u=0$, compatibility gives $v=0$, and the assertion follows from
the prescribed zero continuation. Hence assume $u\ne0$.
Choose $\alpha\in[b^{-1},a^{-1}]$ such that
$v=(I-\alpha H)u$.  With
$x_0=H^{-1}u$ and $x_1=H^{-1}v$, one has
\[
  x_1=x_0-\alpha u.
\]
Thus $u\mapsto v$ is a legal warm-up gradient step.  In the BB1 case,
if $s_0=x_1-x_0$ and $y_0=v-u$, then
\[
  s_0=-\alpha u,
  \qquad y_0=-\alpha Hu,
  \qquad
  \frac{s_0^Ts_0}{s_0^Ty_0}=\alpha_I(u),
\]
so the orbit is a genuine BB1 trajectory.  For a general fixed weight,
the next stepsize is $\alpha_W(u)$ by definition, and hence
\[
  (u,v)\mapsto
  \bigl(v,(I-\alpha_W(u)H)v\bigr)=T_W(u,v).
\]
The same identity repeats until termination, after which both the
algorithm and the map are continued by zeros.
\end{proof}

\begin{theorem}\label{thm:weighted-BB-rate}
Fix $H\succ0$.  There exist constants $C_H\ge1$ and
$\rho_H\in(0,1)$, depending only on $H$, such that every positive
spectral weight $W=\omega(H)$ satisfies
\begin{equation}\label{eq:weighted-BB-state-rate}
  \norm{T_W^k(u,v)}_\times
  \le C_H\sqrt{\kappa_2(W)}\,
       \rho_H^k\norm{(u,v)}_\times,
  \qquad
  (u,v)\in\Om_{\rm BB}(H),\quad k\ge0.
\end{equation}
This includes BB1, BB2, and every fixed power weight $W=H^\tau$,
$\tau\in\R$.
\end{theorem}

\begin{proof}
Raydan's convergence theorem shows that every BB1 trajectory on a
strictly convex quadratic either terminates or has gradients tending
to zero \cite{Raydan1993}.  By Lemma~\ref{lem:BB-realization}, every
$T_I$-orbit on $\Om_{\rm BB}(H)$ is such a trajectory.  Hence
$T_I^kz\to0$ for every $z\in\Om_{\rm BB}(H)$, and
Proposition~\ref{prop:BB-map} and \cref{thm:closed-cone} give constants
$C_H\ge1$ and $\rho_H\in(0,1)$, depending only on $H$, for which
\eqref{eq:weighted-BB-state-rate} holds when $W=I$.

For general $W$, Proposition~\ref{prop:BB-conjugacy} gives
$T_W^k=\cS_W^{-1}T_I^k\cS_W$.  The operator norms of $\cS_W$ and
$\cS_W^{-1}$ for the product norm are $\norm{S}_2$ and
$\norm{S^{-1}}_2$, respectively.  Hence
\[
  \norm{T_W^kz}_\times
  \le\norm{S^{-1}}_2\,
      C_H\rho_H^k
      \norm{\cS_W z}_\times
  \le C_H\norm{S^{-1}}_2\norm{S}_2
      \rho_H^k\norm{z}_\times.
\]
Since $S=W^{1/2}$,
\[
  \norm{S^{-1}}_2\norm{S}_2
  =\sqrt{\kappa_2(W)}.
\]
For $W=H^\tau$ the matrix is positive definite for every real $\tau$.
\end{proof}

\begin{corollary}\label{cor:classical-BB}
Consider BB1, BB2, or a fixed weighted rule
\eqref{eq:weighted-BB-recurrence}, with an arbitrary positive finite
first stepsize and zero continuation after termination. Then
\begin{equation}\label{eq:actual-BB-rate}
  \norm{(g_k,g_{k+1})}_\times
  \le C_H\sqrt{\kappa_2(W)}\,
       \rho_H^{\,k-1}
       \norm{(g_1,g_2)}_\times,
  \qquad k\ge1.
\end{equation}
Consequently,
\begin{align}
  \norm{x_k-x_*}
  &\le a^{-1}C_H\sqrt{\kappa_2(W)}\,
       \rho_H^{\,k-1}
       \norm{(g_1,g_2)}_\times,\label{eq:BB-error-rate}\\
  0\le f(x_k)-f(x_*)
  &\le\frac{C_H^2\kappa_2(W)}{2a}\,
       \rho_H^{\,2(k-1)}
       \norm{(g_1,g_2)}_\times^2.\label{eq:BB-gap-rate}
\end{align}
\end{corollary}

\begin{proof}
If $g_0=0$ or $g_1=0$, the conclusion is immediate under zero
continuation. Otherwise the first computed weighted Rayleigh step is
$\alpha_1=\alpha_W(g_0)\in[b^{-1},a^{-1}]$. Hence
$(g_1,g_2)\in\Om_{\rm BB}(H)$ even when the arbitrary first step that
produced $g_1$ does not. Moreover,
\[
  (g_k,g_{k+1})=T_W^{k-1}(g_1,g_2),
  \qquad k\ge1.
\]
Applying \cref{thm:weighted-BB-rate} yields \eqref{eq:actual-BB-rate}.
The remaining estimates follow from
\[
  \norm{x-x_*}\le a^{-1}\norm{g(x)}
\]
and
\[
  f(x)-f(x_*)
  =\frac12g(x)^TH^{-1}g(x)
  \le\frac1{2a}\norm{g(x)}^2.
\]
\end{proof}

\section{Conclusion}\label{sec:conclusion}

The closed-cone construction separates the algorithm-specific input,
pointwise convergence from every compatible state, from a geometric
rate-upgrade step.  For BB1, Raydan's convergence theorem supplies the
input; compactness and homogeneity yield the uniform $R$-linear rate.
Spectral conjugacy then transfers the rate to BB2 and every fixed
positive spectral weight while preserving the homogeneous growth
radius.

The same proof architecture applies to a finite-memory spectral
recurrence whenever its realizable states form a finite-dimensional
closed invariant cone, the induced transition is continuous and
degree-one positively homogeneous, and pointwise stability holds from
every cone state.  The central task is therefore the construction of
the compatible state space.  For the classical BB rules, a sharper
evaluation of the common growth radius is developed in
\cite{YangYuanSharp2026}.

\section*{Acknowledgements}
Shutai Yang thanks Shixiang Chen for supervising his undergraduate thesis at the University of Science and Technology of China, and Xiaowei Xu for supervising his project under the National College Students Innovation and Entrepreneurship Training Program at the same university.  The present article grew out of those two undergraduate projects.

\section*{Declarations}

\paragraph{Funding.}
The work of Shutai Yang was supported by the National College Students Innovation and Entrepreneurship Training Program, administered by the University of Science and Technology of China (Project No.~202510358091).  The work of Ya-xiang Yuan was supported by the National Natural Science Foundation of China (Grant No.~12288201).

\paragraph{Competing interests.}
The authors have no relevant financial or non-financial interests to disclose.

\paragraph{Author contributions.}
Shutai Yang conceived the study, developed the mathematical results, carried out the analysis, and wrote the original draft.  Ya-xiang Yuan provided guidance on the research direction and reviewed the manuscript.  Both authors read and approved the final manuscript.

\paragraph{Declaration on the use of generative AI.}
During the preparation of this manuscript, Shutai Yang used OpenAI language models for language and \LaTeX{} editing, and the presentation of mathematical notation.  He independently checked and revised all model-assisted material.  Both authors reviewed and approved the final manuscript and take full responsibility for its content.

\paragraph{Data availability.}
No datasets were generated or analyzed during the current study.

\paragraph{Code availability.}
No computational code was generated or used in this entirely theoretical study.


\begin{thebibliography}{99}
\small
\setlength{\itemsep}{0.35em}

\bibitem{BarzilaiBorwein1988}
J.~Barzilai and J.~M. Borwein.
Two-point step size gradient methods.
\emph{IMA Journal of Numerical Analysis}, 8(1):141--148, 1988.
\doi{10.1093/imanum/8.1.141}.

\bibitem{Dai2003}
Y.-H.~Dai.
Alternate step gradient method.
\emph{Optimization}, 52(4--5):395--415, 2003.
\doi{10.1080/02331930310001611547}.

\bibitem{DaiLiao2002}
Y.-H.~Dai and L.-Z.~Liao.
$R$-linear convergence of the Barzilai and Borwein gradient method.
\emph{IMA Journal of Numerical Analysis}, 22(1):1--10, 2002.
\doi{10.1093/imanum/22.1.1}.

\bibitem{FriedlanderMartinezMolinaRaydan1999}
A.~Friedlander, J.~M. Mart\'inez, B.~Molina, and M.~Raydan.
Gradient method with retards and generalizations.
\emph{SIAM Journal on Numerical Analysis}, 36(1):275--289, 1999.
\doi{10.1137/S003614299427315X}.

\bibitem{Raydan1993}
M.~Raydan.
On the Barzilai and Borwein choice of steplength for the gradient method.
\emph{IMA Journal of Numerical Analysis}, 13(3):321--326, 1993.
\doi{10.1093/imanum/13.3.321}.

\bibitem{ShenHu2012}
J.~Shen and J.~Hu.
Stability of discrete-time switched homogeneous systems on cones and conewise homogeneous inclusions.
\emph{SIAM Journal on Control and Optimization}, 50(4):2216--2253, 2012.
\doi{10.1137/110845215}.

\bibitem{Tuna2008}
S.~E. Tuna.
Growth rate of switched homogeneous systems.
\emph{Automatica}, 44(11):2857--2862, 2008.
\doi{10.1016/j.automatica.2008.03.017}.

\bibitem{YangYuanSharp2026}
S.~Yang and Y.-X.~Yuan.
The sharp worst-case asymptotic rate of the Barzilai--Borwein method in $\mathbb R^d$ and Hilbert spaces.
Preprint, \href{https://arxiv.org/abs/2608.07839}{arXiv:2608.07839}, 2026. 
\doi{10.48550/arXiv.2608.07839}.

\end{thebibliography}
\end{document}